\documentclass[english,4paper,10pt]{article}
\usepackage[utf8]{inputenc}
\usepackage[T1]{fontenc}
\usepackage[english]{babel}
\usepackage[autostyle=true]{csquotes}
\usepackage{mathdots}
\usepackage{latexsym, amssymb, amsmath}
\usepackage{amsthm}
\makeatletter
\renewcommand*\env@matrix[1][*\c@MaxMatrixCols c]{%
	\hskip -\arraycolsep
	\let\@ifnextchar\new@ifnextchar
	\array{#1}}
\usepackage{url}
\usepackage{float}
\usepackage{caption}
\usepackage{subcaption}
\usepackage{color}   
\usepackage{hyperref}
\usepackage[symbol]{footmisc}
\hypersetup{
	colorlinks=true, 
	linkcolor=black,
	filecolor=magenta,      
	urlcolor=cyan,
	linktoc=all,     
}
\usepackage[ruled,vlined]{algorithm2e}
\usepackage{graphicx}
\graphicspath{ {./images/} }
\usepackage{amscd} 
\usepackage{tikz-cd}
\usepackage[all,cmtip]{xy}
\usepackage{graphicx}

\newcommand{\R}{\mathbb{R}}

\newcommand{\F}{\mathbb{F}}

\newtheorem{thm}{Theorem}[section]

\newtheorem{corollary}[thm]{Corollary}
\newtheorem{prop}[thm]{Proposition}

\theoremstyle{definition}
\newtheorem{definition}[thm]{Definition}
\newtheorem{remark}{Remark}[section]

\makeatletter
\def\blfootnote{\xdef\@thefnmark{}\@footnotetext}
\makeatother

\date{}

\usepackage[%
backend=bibtex,%
style=numeric-comp,%
sorting=nyt,%
hyperref,%
]{biblatex}

\AtBeginBibliography{}

\renewbibmacro{in:}{%
	\ifentrytype{article}
	{}
	{\bibstring{in}%
		\printunit{\intitlepunct}}}

\DeclareUnicodeCharacter{2212}{-}
\begin{document}
	\sloppy
	
	\title{Non-nilpotent Leibniz algebras with\\one-dimensional derived subalgebra\blfootnote{Keywords: Leibniz algebra, Lie algebra, Derivation, Biderivation, Coquecigrue problem.} \blfootnote{\textit{\textup{2020} Mathematics Subject Classification}: 16W25, 17A32, 17B30, 17B40, 20M99, 22A30.} 
	\blfootnote{The authors are supported by University of Palermo and by the “National Group for Algebraic and Geometric Structures, and their Applications” (GNSAGA – INdAM). The first author is also supported by \textbf{UNIPA FFR 2024 VQR 2024}. The second and third authors are also supported by the National Recovery and Resilience Plan (NRRP), Mission 4, Component 2, Investment 1.1, Call for tender No.\ 1409 published on 14/09/2022 by the Italian Ministry of University and Research (MUR), funded by the European Union -- NextGenerationEU -- Project Title \textbf{Quantum Models for Logic, Computation and Natural Processes} (\textbf{QM4NP}) -- CUP \textbf{B53D23030160001} -- Grant Assignment Decree No.\ 1371 adopted on 01/09/2023 by the Italian Ministry of Ministry of University and Research (MUR), and by the \textbf{Sustainability Decision Framework} (\textbf{SDF}) Research Project --  CUP \textbf{B79J23000540005} -- Grant Assignment Decree No.\ 5486 adopted on 04/08/2023. }} \maketitle
	\noindent
	{{Alfonso Di Bartolo, 
		{Gianmarco La Rosa}, {Manuel Mancini} \\ \\
		\footnotesize{Dipartimento di Matematica e Informatica}\\
		\footnotesize{Universit\`a degli Studi di Palermo, Via Archirafi 34, 90123 Palermo, Italy}\\
		\footnotesize{alfonso.dibartolo@unipa.it}, ORCID: 0000-0001-5619-2644 \\
		\footnotesize{gianmarco.larosa@unipa.it}, ORCID: 0000-0003-1047-5993 \\
		\footnotesize{manuel.mancini@unipa.it}, ORCID: 0000-0003-2142-6193}

\begin{abstract}
In this paper we study non-nilpotent non-Lie Leibniz $\F$-algebras with one-dimensional derived subalgebra, where $\F$ is a field with $\operatorname{char}(\F) \neq 2$. We prove that such an algebra is isomorphic to the direct sum of the two-dimensional non-nilpotent non-Lie Leibniz algebra and an abelian algebra. We denote it by $L_n$, where $n=\dim_\F L_n$. This generalizes the result found in \cite{solvLeib}, which is only valid when $\F=\mathbb{C}$ . Moreover, we find the Lie algebra of derivations, its Lie group of automorphisms and the Leibniz algebra of biderivations of $L_n$. Eventually, we solve the \emph{coquecigrue problem} for $L_n$ by integrating it into a Lie rack.
\end{abstract}

\section*{Introduction}

Leibniz algebras were introduced by J.-L.\ Loday in \cite{loday1993version} as a non-skew symmetric version of Lie algebras. Earlier such algebraic structures were also considered by A.\ Blokh, who called them D-algebras \cite{blokhLie} for their strict connection with derivations. Leibniz algebras play a significant role in different areas of mathematics and physics.

Many results of Lie algebras are still valid for Leibniz algebras. One of them is the \emph{Levi decomposition}, which states that any Leibniz algebra over a field $\F$ of characteristic zero is the semidirect sum of its radical and a semisimple Lie algebra. This makes clear the importance of the problem of classification of solvable and nilpotent Lie / Leibniz algebras, which has been dealt with since the early 20th century (see \cite{bartolone2011nilpotent}, \cite{BarDiFal18}, \cite{Falcone2017class}, \cite{6dimLeib}, \cite{4dimLeib}, \cite{solvLeib}, \cite{Katgorodov1} and \cite{5dimLeib}, just for giving a few examples).\\

In \cite{LM2022} and \cite{LaRosaMancini2} nilpotent Leibniz algebras $L$ with one-dimensional derived subalgebra $[L,L]$ were studied and classified. It was proved that, up to isomorphism, there are three classes of \emph{indecomposable} Leibniz algebras with these properties, namely the \emph{Heisenberg} algebras $\mathfrak{l}_{2n+1}^A$, which are parameterized by their dimension $2n+1$ and by a matrix $A$ in canonical form, the \emph{Kronecker} algebra $\mathfrak{k}_n$ and the \emph{Dieudonné} algebra $\mathfrak{d}_n$, both parameterized by their dimension only. We want to complete this classification by studying non-nilpotent Leibniz $\F$-algebras with one-dimensional derived subalgebra, where $\F$ is a field with $\operatorname{char}(\F) \neq 2$. Using the theory of non-abelian extensions of Leibniz algebras introduced in \cite{NonAbExt}, we prove that a non-nilpotent non-Lie Leibniz algebra $L$ with $\dim_\F L=n$ and $\dim_\F [L,L]=1$ is isomorphic to the direct sum of the two-dimensional non-nilpotent non-Lie Leibniz algebra $S_2$, i.e.\ the algebra with basis $\lbrace e_1,e_2 \rbrace$ and multiplication table given by $[e_2,e_1]=e_1$, and an abelian algebra of dimension $n-2$. We denote it by $L_n$. This generalizes the result found in Theorem 2.6 of \cite{solvLeib}, where the authors proved that a \emph{complex} non-split non-nilpotent non-Lie Leibniz algebra with one-dimensional derived subalgebra is isomorphic to $S_2$.

We study in detail the properties of the algebra $L_n$ and we compute the Lie algebra of derivations $\operatorname{Der}(L_n)$, its Lie group of automorphism $\operatorname{Aut}(L_n$) and the Leibniz algebra of biderivations $\operatorname{Bider}(L_n)$.

Finally, we solve the \emph{coquecigrue problem} for the Leibniz algebra $L_n$. We mean the problem, formulated by J.-L.\ Loday in \cite{loday1993version}, of finding a generalization of Lie third theorem to Leibniz algebras. Using M.\ K.\ Kinyon's results for the class of real \emph{split Leibniz algebras} (see \cite{kinyon2004leibniz}), we show how to explicitly integrate $L_n$ into a Lie rack defined over the vector space $\R^n$.

\section{Preliminaries}

We assume that $\F$ is a field with $\operatorname{char}(\F)\neq2$. For the general theory we refer to \cite{ayupov2019leibniz}.

\begin{definition}
	A \emph{left Leibniz algebra} over $\mathbb{F}$ is a vector space $L$ over $\mathbb{F}$ endowed with a bilinear map (called $commutator$ or $bracket$) $\left[-,-\right]\colon L\times L \rightarrow L$ which satisfies the \emph{left Leibniz identity}
	$$
	\left[x,\left[y,z\right]\right]=\left[\left[x,y\right],z\right]+\left[y,\left[x,z\right]\right], \quad \forall x,y,z\in L.
	$$
\end{definition}

	In the same way we can define a right Leibniz algebra, using the right Leibniz identity
	$$
	\left[\left[x,y\right],z\right]=\left[[x,z],y\right]+\left[x,\left[y,z\right]\right], \quad \forall x,y,z\in L.
	$$
	Given a left Leibniz algebra $L$, the multiplication $[x,y]^{\operatorname{op}}=[y,x]$ defines a right Leibniz algebra structure on $L$. 
	
	A Leibniz algebra that is both left and right is called \emph{symmetric Leibniz algebra}. From now on we assume that $\dim_\F L<\infty$.\\
	 
	We have a full inclusion functor $i\colon \textbf{Lie} \rightarrow \textbf{Leib}$ that embeds Lie algebras over $\F$ into Leibniz algebras over $\F$. 
	 Its left adjoint is the functor $\pi\colon \textbf{Leib} \rightarrow \textbf{Lie}$, which associates to each Leibniz algebra $L$ the quotient $L/\operatorname{Leib}(L)$, where $\operatorname{Leib}(L)$ is the smallest bilateral ideal of $L$ such that the quotient $L/\operatorname{Leib}(L)$ becomes a Lie algebra. $\operatorname{Leib}(L)$ is defined as the subalgebra generated by all elements of the form $\left[x,x\right]$, for any $x\in L$, and it is called the \emph{Leibniz kernel} of $L$.
	 
	We define the left and the right center of a Leibniz algebra
	$$
	\operatorname{Z}_l(L)=\left\{x\in L \,|\,\left[x,L\right]=0\right\},\,\,\, \operatorname{Z}_r(L)=\left\{x\in L \,|\,\left[L,x\right]=0\right\}.
	$$
	The intersection of the left and right center is called the \emph{center} of $L$ and it is denoted by $\operatorname{Z}(L)$. In general for a left Leibniz algebra $L$, the left center $\operatorname{Z}_l(L)$ is a bilateral ideal, meanwhile the right center is not even a subalgebra. Furthermore, one can check that $\operatorname{Leib}(L) \subseteq \operatorname{Z}_l(L)$.
		
	The definition of derivation for a Leibniz algebra is the same as in the case of Lie algebras. 
	
	\begin{definition}
		A linear map $d \colon L \rightarrow L$ is a \emph{derivation} of $L$ if
		$$
		d([x,y])=[d(x),y]+[x,d(y)], \quad \forall x,y \in L.
		$$
	\end{definition}
	
	An equivalent way to define a left Leibniz algebra $L$ is to saying that the left adjoint maps $\operatorname{ad}_x=[x,-]$ are derivations. Meanwhile the right adjoint maps $\operatorname{Ad}_x=[-,x]$ are not derivations in general. The set $\operatorname{Der}(L)$ of all derivations of $L$ is a Lie algebra with the usual bracket $[d,d']=d \circ d' - d' \circ d$ and the set $\operatorname{Inn}(L)$ spanned by the left adjoint maps, which are called \emph{inner derivations}, is an ideal of $\operatorname{Der}(L)$. Moreover $\operatorname{Aut}(L)$ is a Lie group and its Lie algebra is precisely $\operatorname{Der}(L)$.
	
	In \cite{loday1993version} J.-L.\ Loday introduced the notion of anti-derivation and biderivation for a Leibniz algebra.
	
	\begin{definition}
		A linear map $D \colon L \rightarrow L$ is an \emph{anti-derivation} of $L$ if
		$$
		D([x,y])=[x,D(y)]-[y,D(x)], \quad \forall x,y \in L.
		$$
	\end{definition}

	The space $\operatorname{ADer}(L$)  of anti-derivations of $L$ has a $\operatorname{Der}(L)$-module structure with the extra multiplication $d \cdot D = d \circ D - D \circ d$, for any derivation $d$ and for any anti-derivation $D$, and one can check that the right adjoint maps $\operatorname{Ad}_x$ are anti-derivations.

\begin{definition}
A \emph{biderivation} of $L$ is a pair $(d,D) \in \operatorname{Der}(L) \times \operatorname{ADer}(L)$ such that
$$
[d(x)+D(x),y]=0, \quad \forall x,y \in L.
$$
\end{definition}

The set $\operatorname{Bider}(L)$ of all biderivations of $L$ has a Leibniz algebra structure with the bracket
$$
[(d,D),(d',D')]=([d,d'],d \cdot D')
$$
and it is defined a Leibniz algebra homomorphism
$$
L \rightarrow \operatorname{Bider}(L), \; x \mapsto (\operatorname{ad}_x,\operatorname{Ad}_x).
$$
The pair $(\operatorname{ad}_x,\operatorname{Ad}_x)$ is called the \emph{inner biderivation} associated with $x \in L$ and the set of all inner biderivations of $L$ forms a Leibniz subalgebra of $\operatorname{Bider}(L)$.

We recall the definitions of solvable and nilpotent Leibniz algebras.
	
	\begin{definition}
		Let $L$ be a left Leibniz algebra over $\F$ and let
		$$
		L^{0}=L,\,\,L^{k+1}=[L^{k},L^k],\quad \forall k\geq0
		$$ be the \emph{derived series of $L$}. $L$ is \emph{$n-$step solvable} if $L^{n-1}\neq0$ and $L^{n}=0$.
	\end{definition}
	\begin{definition}
		Let $L$ be a left Leibniz algebra over $\F$ and let
		$$
		L^{(0)}=L,\,\,L^{(k+1)}=[L,L^{(k)}],\quad \forall k\geq0
		$$ be the \emph{lower central series of $L$}. $L$ is \emph{$n-$step nilpotent} if $L^{(n-1)}\neq0$ and $L^{(n)}=0$.
	\end{definition}
	
	When $L$ is two-step nilpotent, it lies in different varieties of non-associative algebras, such as associative, alternative and Zinbiel algebras. In this case we refer at $L$ as a \emph{two-step nilpotent algebra} and we have the following.

	\begin{prop}{\ }
		\begin{enumerate}
			\item[(i)] 	If $L$ is a two-step nilpotent algebra, then $L^{(1)}=[L,L]\subseteq \operatorname{Z}(L)$ and $L$ is a symmetric Leibniz algebra.
			\item[(ii)] If $L$ is a left nilpotent Leibniz algebra with $\dim_\F [L,L]=1$, then  $L$ is two-step nilpotent.
		\end{enumerate}
	\end{prop}
	
	In \cite{LM2022} the classification of nilpotent Leibniz algebras with one-dimensional derived subalgebra was established. The classification revealed that, up to isomorphism, there exist only three classes of indecomposable nilpotent Leibniz algebras of this type.
	
	\begin{definition}\cite{LM2022}
		Let $f(x)\in\F\left[x\right]$ be a monic irreducible polynomial. Let $k\in\mathbb{N}$ and let $A=(a_{ij})_{i,j}$ be the companion matrix of $f(x)^k$. The \emph{Heisenberg} algebra $\mathfrak{l}_{2n+1}^A$ is the $(2n+1)$-dimensional Leibniz algebra with basis $\left\{e_1,\ldots,e_n,f_1,\ldots,f_n,z\right\}$ and the brackets are given by
		$$
		[e_i,f_j]=(\delta_{ij}+a_{ij})z, \quad [f_j,e_i]=(-\delta_{ij}+a_{ij})z, \quad \forall i,j=1,\ldots,n.
		$$
	\end{definition}
	
	 When $A$ is the zero matrix, then we obtain the $(2n+1)-$dimensional Heisenberg Lie algebra $\mathfrak{h}_{2n+1}$.
	
	\begin{definition}\cite{LM2022}
		Let $n \in \mathbb{N}$. The \emph{Kronecker} algebra $\mathfrak{k}_{n}$ is the $(2n+1)$-dimensional Leibniz algebra with basis $\left\{e_1,\ldots,e_n,f_1,\ldots,f_n,z\right\}$ and the brackets are given by
		\begin{align*}
			&[e_i,f_i] = [f_i,e_i] = z, \quad \forall i=1,\ldots,n \\
			& [e_i,f_{i-1}] = z, [f_{i-1},e_i] = −z, \quad \forall i= 2,\ldots,n.
		\end{align*}
	\end{definition}
	
	\begin{definition}\cite{LM2022}
		Let $n \in \mathbb{N}$. The \emph{Dieudonné} algebra $\mathfrak{d}_n$ is the $(2n+2)$-dimensional Leibniz algebra with basis $\left\{e_1,\ldots,e_{2n+1},z\right\}$ and the brackets are given by
		\begin{align*}
			&\left[e_1,e_{n+2}\right]=z,\\
			&\left[e_i,e_{n+i}\right]=\left[e_i,e_{n+i+1}\right]=z, \quad \forall i=2,\ldots,n,\\
			&\left[e_{n+1},e_{2n+1}\right]=z,\\
			&\left[e_i,e_{i-n}\right]=z, \quad \left[e_i,e_{i-n-1}\right]=-z, \quad \forall i=n+2,\ldots,2n+1.
		\end{align*}
		
	\end{definition}
	
	We want to extend this classification by studying non-nilpotent Leibniz algebras with one-dimensional derived subalgebra.
	
	\section{Non-nilpotent Leibniz algebras with one-dimensional derived subalgebra}
	
	Let $L$ be a non-nilpotent left Leibniz algebra over $\F$ with $\dim_\F L=n$ and $\dim_\F \left[L,L\right]=1$. We observe that such an algebra is two-step solvable since the derived subalgebra $\left[L,L\right]$ is abelian.
	
	It is well known that a non-nilpotent Lie algebra with one-dimensional derived subalgebra is isomorphic to the direct sum of the two-dimensional non-abelian Lie algebra and an abelian algebra (see \cite[Section 3]{erdmann2006introduction}). Thus we are interested in the classification of non-Lie Leibniz algebras with these properties.
	
	In \cite[Theorem 2.6]{solvLeib} the authors prove that a \emph{complex} non-split non-nilpotent non-Lie Leibniz algebra with one-dimensional derived subalgebra is isomorphic to the two-dimensional algebra with basis $\lbrace e_1,e_2 \rbrace$ and multiplication table $[e_2,e_1]=[e_2,e_2]=e_1$. Here we generalize this result when $\F$ is a general field with $\operatorname{char}(\F) \neq 2$.
	
	\begin{prop}
	Let $L$ be a non-nilpotent left Leibniz algebra over $\F$ with $\dim_\F \left[L,L\right]=1$. Then $L$ has a two-dimensional bilateral ideal $S$ which is isomorphic to one of the following Leibniz algebras:
	\begin{itemize}

	    \item[(i)] $S_1=\langle e_1,e_2 \rangle$ with  $ \left[e_2,e_1\right]=-\left[e_1,e_2\right]=e_1$;
	    \item[(ii)] $S_2=\langle e_1,e_2 \rangle$ with  $\left[e_2, e_1\right]=\left[e_2,e_2\right]=e_1$.
	    \end{itemize}

	\end{prop}
	
	\begin{proof}
	Let $\left[L,L\right]=\F z$. $L$ is not nilpotent, then
	\[
	[L,\left[L,L\right]]\neq0,
	\]
	i.e.\ $z \notin \mathrm{Z}_r(L)$. Since $\left[L,L\right]$ is an abelian algebra, there exists a vector $x\in L$, which is linearly independent than $z$, such that $\left[x, z\right]\neq0$. Thus
	\[
	\left[x, z\right]=\gamma z,
	\] 
	for some $\gamma\in\F^\ast$. The subspace $S=\langle x, z\rangle$ is an ideal of $L$ and it is not nilpotent: in fact
	\[
	0\neq\gamma z=\left[x,z\right]\in\left[S,\left[S,S\right]\right].
	\]
	Thus $S$ is a non-nilpotent Leibniz algebra. Using the classification of two-dimensional Leibniz algebras given by C.\ Cuvier in \cite{2-dimCuvier}, $S$ is isomorphic either to $S_1$ or to $S_2$.
	\end{proof}

	\begin{remark}
	The algebras $S_1$ and $S_2$ are respectively the Leibniz algebras $L_2$ and $L_4$ of Section 3.1 in \cite{ayupov2019leibniz}. We observe that $S_1$ is a Lie algebra, meanwhile $S_2$ is a non-right left Leibniz algebra.
	\end{remark}

     One can see $L$ as an extension of the abelian algebra $L_0 = L/S \cong \F^{n-2}$ by $S$ \cite{NonAbExt}
     \begin{equation}\label{ext}
     	\xymatrix{0 \ar[r] & S \ar[r]^-i & L \ar@<.5ex>[r]^-\pi & L_0 \ar@<.5ex>[l]^-s \ar[r] & 0}.
     \end{equation}
     It turns out that there exists an equivalence of Leibniz algebra extensions
      \begin{equation*}
     	\begin{tikzcd} 
     		0\ar[r]
     		&S \arrow [r, "i_2 "] \ar[d, "\operatorname{id}_S"']
     		&L_0 \ltimes_{\omega} S \arrow[r, shift left, "\pi_1"] \ar[d, "\theta"]& 
     		L_0 \ar[r]\ar[l, shift left, "i_1"]\ar[d, "\operatorname{id}_{B}"]
     		&0 \\
     		0\ar[r]
     		&S \arrow [r, "i"]
     		&L \arrow[r, shift left, "\pi"] & 
     		L_0 \ar[r]\ar[l, shift left, "\sigma"]
     		&0 
     	\end{tikzcd}
     \end{equation*}
     where $L_0 \ltimes_{\omega} S$ is the Leibniz algebra defined on the direct sum of vector spaces $L_0 \oplus S$ with the bilinear operation given by
     $$
     [(x,a),(y,b)]_{(l,r,\omega)}=(0,[a,b]+l_x(b)+r_y(a)+\omega(x,y)),
     $$
     where 
     $$
     \omega(x,y)=[\sigma(x),\sigma(y)]_L-\sigma([x,y]_{L_0})=[\sigma(x),\sigma(y)]_L
     $$
     is the Leibniz algebra $2$-cocycle associated with \eqref{ext} and 
     $$
     l_x(b)=[\sigma(x),i(b)]_L, \; \; r_y(a)=[i(a),\sigma(y)]_L
     $$
     define the action of $L_0$ on $S$; $i_1,i_2,\pi_1$ are the canonical injections and projection. The Leibniz algebra isomorphism $\theta$ is defined by $\theta(x,a)= \sigma(x)+i(a)$, for every $(x,a) \in L_0 \oplus S$. 
     
     By \cite[Proposition 4.2]{NonAbExt}, the $2$-cocycle $\omega \colon L_0 \times L_0 \rightarrow S$ and the linear maps $l,r \colon L_0 \rightarrow \operatorname{gl}(S)$ must satisfy the following set of equations
     \begin{enumerate}
     	\item[(L1)] $l_x([a,b])=[l_x(a),b]+[x,l_x(b)]$;
     	\item[(L2)] $r_x([a,b])=[a,r_x(b)]-[b,r_x(a)]$;
     	\item[(L3)] $[l_x(a)+r_x(a),b]=0$;
     	\item[(L4)] $[l_x,l_y]_{\operatorname{gl}(S)}-l_{[x,y]_{L_0}}=\operatorname{ad}_{\omega(x,y)}$;
     	\item[(L5)] $[l_x,r_y]_{\operatorname{gl}(S)}-r_{[x,y]_{L_0}}=\operatorname{Ad}_{\omega(x,y)}$;
     	\item[(L6)] $r_y(r_x(a)+l_x(a))=0$;
     	\item[(L7)] $l_x(\omega(y,z))-l_y(\omega(x,z))-r_z(\omega(x,y))=$ \\ $=\omega([x,y]_{L_0},z)-\omega(x,[y,z]_{L_0})+\omega(y,[x,z]_{L_0})$
     \end{enumerate}
 for any $x,y \in L_0$ and for any $a,b \in S$. Notice that these equations where also studied in \cite{CigoliManciniMetere} in the case of Leibniz algebra \emph{split extensions}.
 
 \begin{remark}
 	The first three equations state that the pair $(l_x,r_x)$ is a biderivation of the Leibniz algebra $S$, for any $x \in L_0$. Biderivations of low-dimensional Leibniz algebras were classified in \cite{Mancini} and it turns out that
 	\begin{itemize}
 		\item $\operatorname{Bider}(S_1)=\left\lbrace  (d,-d) \; | \; d \in \operatorname{Der}(S_1) \right\rbrace$ and \\ \\
 		$
 		\operatorname{Der}(S_1)=\left\lbrace \begin{pmatrix}
 			\alpha & \beta \\
 			0 & 0 \\
 		\end{pmatrix} \Bigg\vert \; \alpha,\beta \in \mathbb{F} \right\rbrace;
 		$
 		\item $\operatorname{Bider}(S_2)=\left\lbrace \left( \begin{pmatrix}
 			\alpha & \alpha\\
 			0 & 0 \\
 		\end{pmatrix},\begin{pmatrix}
 			0 & \beta\\
 			0 & 0 \\
 		\end{pmatrix} \right)  \Bigg\vert \; \alpha,\beta \in \mathbb{F} \right\rbrace$.
 	\end{itemize}
 \end{remark}

We study now in detail the non-abelian extension \eqref{ext} in both cases that $S$ is isomorphic either to $S_1$ or to $S_2$.

\subsection{$S$ is a Lie algebra}

When $S \cong S_1$, we have that $r_y=-l_y$, for any $y \in L_0$ and the bilinear operation of $L_0 \ltimes_{\omega} S_1$ becomes 
$$
[(x,a),(y,b)]_{(l,\omega)}=(0,[a,b]+l_x(b)-l_y(a)+\omega(x,y)).
$$
The linear map $l_x$ is represented by a $2 \times 2$ matrix
$$
\begin{pmatrix}
	\alpha_x & \beta_x \\
	0 & 0 \\
\end{pmatrix}
$$
with $\alpha_x$,$\beta_x \in \F$. From equations (L4)-(L5) it turns out that
$$
\omega(x,y)=(\alpha_x \beta_y - \alpha_y \beta_x)e_1, \; \; \forall x,y \in L_0
$$
and the $2$-cocycle $\omega$ is skew-symmetric. Moreover, equations (L6)-(L7) are automatically satisfied and the resulting algebra $L_0 \ltimes_{\omega} S_1 \cong L$ is a Lie algebra. We conclude that $L$ is isomorphic to the direct sum of $S_1$ and $L_0 \cong \F^{n-2}$.

\subsection{$S$ is not a Lie algebra}

With the change of basis $e_2 \mapsto e_2 - e_1$, $S_2$ becomes the Leibniz algebra with basis $\lbrace e_1,e_2 \rbrace$ and the only non-trivial bracket given by $[e_2,e_1]=e_1$. Now a biderivation of $S_1$ is represented by a pair of matrices
$$
\left( \begin{pmatrix}
	\alpha & 0\\
	0 & 0 \\
\end{pmatrix},\begin{pmatrix}
	0 & \beta\\
	0 & 0 \\
\end{pmatrix} \right) 
$$
with $\alpha, \beta \in \F$ and the pair $(l_x,r_x) \in \operatorname{Bider}(S_2)$ is defined by $l_x(e_1)=\alpha_x e_1$ and $r_x(e_2)=\beta_x e_1$, for any $x \in L_0$.

Equation (L4) states that $\left[l_x,l_y\right]_{\operatorname{gl}(S_2)}=\left[\omega(x,y), -\right]$, with
\begin{align*}
\left[l_x, l_y\right]_{\operatorname{gl}(S_2)}=l_x\circ l_y-l_y\circ l_x&=
\begin{pmatrix}
\alpha_x & 0 \\
0 & 0
\end{pmatrix}
\begin{pmatrix}
\alpha_y & 0 \\
0 & 0
\end{pmatrix}-
\begin{pmatrix}
	\alpha_y & 0 \\
	0 & 0
\end{pmatrix}
\begin{pmatrix}
	\alpha_x & 0 \\
	0 & 0
\end{pmatrix}=\\
&=\begin{pmatrix}
	\alpha_x\alpha_y & 0 \\
	0 & 0
\end{pmatrix}-
\begin{pmatrix}
	\alpha_x\alpha_y & 0 \\
	0 & 0
\end{pmatrix}=
\begin{pmatrix}
	0 & 0 \\
	0 & 0
\end{pmatrix},
\end{align*}
for any $x,y\in L_0$. Thus $\omega(x,y)\in\mathrm{Z}_l(S_2)=\mathbb{F}e_1$.

From equation (L5) we have $\left[l_x,r_y\right]_{\operatorname{gl}(S_2)}=\left[-,\omega(x,y)\right]_{S_2}$, with
    \[
    	\left[l_x, r_y\right]_{\operatorname{gl}(S_2)}=l_x\circ r_y-r_y\circ l_x=
    	\begin{pmatrix}
    		0 & \alpha_x\beta_y\\
    		0 & 0
    	\end{pmatrix}-
    	\begin{pmatrix}
    		0 & 0 \\
    		0 & 0
    	\end{pmatrix}=
    	\begin{pmatrix}
    		0 & \alpha_x\beta_y\\
    		0 & 0
    	\end{pmatrix}.
    \]
     Thus, for every $a=a_1e_1 + a_2 e_2 \in S_2$ and for every  $x,y\in L_0$, we have
     \[
     [a,\omega(x,y)]=\left[l_x, r_y\right](a)=\alpha_x\beta_ya_2 e_1,
     \]
	i.e.\ $\omega(x,y)=\alpha_x \beta_y e_1$. Finally, equations (L6) and (L7) are identically satisfied.
	
Summarizing we have
\begin{equation*}
	\begin{cases}
		l_x\equiv\begin{pmatrix}
			\alpha_x & 0 \\
			0 & 0
		\end{pmatrix}\\\\
		r_y\equiv\begin{pmatrix}
			0 & \beta_y \\
			0 & 0 
		\end{pmatrix}\\\\
		\omega(x,y)=\alpha_x\beta_ye_1
	\end{cases}
\end{equation*}
for every $x,y\in L_0$ and the bilinear operation $[-,-]_{(l,r,\omega)}$ becomes
$$
[(x,a),(y,b)]_{(l,r,\omega)}=(0,(a_2 b_1+\alpha_x b_1 + \beta_y a_2 + \alpha_x \beta _y)e_1),
$$
for any $x$, $y \in L_0$ and for any $a=a_1 e_1 + a_2 e_2$, $b=b_1 e_1 + b_2 e_2 \in S_2$.

If we fix a basis $\lbrace f_3,\ldots,f_n\rbrace$ of $L_0$ and we denote by
\[
\alpha_i=\alpha_{f_i}, \; \; \beta_i=\beta_{f_i}, \; \; \forall i=3,\ldots,n
\]
then $L$ is isomorphic to the Leibniz algebra with basis $\left\{e_1, e_2, f_3, \ldots, f_n\right\}$ and non-zero brackets
	\begin{align*}
		&\left[e_2,e_1\right]=e_1 \\
		&\left[e_2, f_i\right]=\beta_ie_1, \quad\forall i=3,\ldots, n\\
		&\left[f_i, e_1\right]=\alpha_ie_1, \quad\forall i=3,\ldots, n\\
		&\left[f_i, f_j\right]=\alpha_i\beta_je_1, \quad \forall i,j=3,\ldots, n.
	\end{align*}
	

With the change of basis $f_i \mapsto f_i' = \dfrac{f_i}{\beta_i}-e_1$, if $\beta_i \neq 0$, we obtain that
\begin{align*}
	&[e_2,f_i']=e_1-[e_2,e_1]=0,\\
	&[f_i',e_1]=\gamma_i e_1, \; \; \text{where } \gamma_i = \frac{\alpha_i}{\beta_i},\\
	&[f_i,f_j']=\alpha_i e_1 - [f_i,e_1]=0,\\
	&[f_i',f_j']=\gamma_i e_1 - \frac{1}{\beta_i} [f_i,e_1]=0.
\end{align*}
If we denote again $f_i \equiv f_i'$ and $\alpha_i \equiv \gamma_i$ when $\beta_i \neq 0$, then $L$ has basis $\lbrace e_1,e_2,f_3,\ldots,f_n \rbrace$ and non-trivial brackets
\[
[e_2,e_1]=e_1, \; \; [f_i,e_1]=\alpha_i e_1, \; \; \forall i=3,\ldots,n.
\]
Finally, when $\alpha_i \neq 0$, we can operate the change of basis
\[
f_i \mapsto \dfrac{f_i}{\alpha_i} - e_2.
\]
One can check that the only non-trivial bracket now is $[e_2,e_1]=e_1$ and $L$ is isomorphic to the direct sum of $S_2$ and the abelian algebra $L_0 \cong \F^{n-2}$. This allows us to conclude with the following.

\begin{thm}
	Let $\F$ be a field with $\operatorname{char}(\F) \neq 2$. Let $L$ be a non-nilpotent non-Lie left Leibniz algebra over $\F$ with $\dim_\F L= n$ and $\dim_\F [L,L]=1$. Then $L$ is isomorphic to the direct sum of the two-dimensional non-nilpotent non-Lie Leibniz algebra $S_2$ and an abelian algebra of dimension $n-2$. We denote this algebra by $L_n$.
	\qed
\end{thm}

If we suppose that $L$ is a \emph{non-split} algebra, i.e.\ $L$ cannot be written as the direct sum of two proper ideals, then we obtain the following result, that is a generalization of \cite[Theorem 2.6]{solvLeib} and which is valid over a general field $\F$ with $\operatorname{char}(\F) \neq 2$.

\begin{corollary}
	Let $L$ be a non-split non-nilpotent non-Lie left Leibniz algebra over $\F$ with $\dim_\F L= n$ and $\dim_\F [L,L]=1$. Then $n=2$ and $L\cong S_2$.
	\qed
\end{corollary}

Now we study in detail the algebra $L_n=S_2 \oplus \F^{n-2}$ by describing the Lie algebra of derivations, its Lie group of automorphisms and the Leibniz algebra of biderivations. Moreover, when $\F=\R$, we solve the \emph{coquegigrue problem} (see \cite{covez2013local} and \cite{kinyon2004leibniz}) for $L_n$ by integrating it into a Lie rack.

	\subsection{Derivations, automorphisms and biderivations of $L_n$}
	
	Let $n \geq 2$ and let $L_n=S_2 \oplus \F^{n-2}$. We fix the basis $\mathcal{B}_n=\lbrace e_1,e_2,f_3,\ldots,f_n \rbrace$ of $L_n$ and we recall that the only non-trivial commutator is $[e_2,e_1]=e_1$. A straightforward application of the algorithm proposed in \cite{Mancini} for finding derivations and anti-derivations of a Leibniz algebra as pair of matrices with respect to a fixed basis produces the following.
	
	\begin{thm}{\ }
		\begin{enumerate}
			\item[(i)] A derivation of $L_n$ is represented, with respect to the basis $\mathcal{B}_n$, by a matrix
			\begin{equation}
				$$\[
			\left(
			\begin{array}{c | c} 
			\begin{array}{c c} 
			\alpha & 0 \\
			0 & 0 \\
			\end{array}  & 	\begin{array}{c c c c} 
					0 & 0 & \cdots & 0 \\
					0 & 0 & \cdots & 0 \\
				\end{array} \\ 
				\hline 
					\begin{array}{c c} 
						0 & a_3 \\
						0 & a_4 \\
						\vdots & \vdots \\
						0 & a_n \\
				\end{array} 
			& A 
			\end{array} 
			\right)
				\]$$
			\end{equation}
		where $A \in \operatorname{M}_{n-2}(\F)$.
			
			\item[(ii)] The group of automorphisms $\operatorname{Aut}(L_n)$ is the Lie subgroup of $\operatorname{GL}_n(\F)$ of matrices of the form
			\begin{equation}
				$$\[
				\left(
				\begin{array}{c | c} 
					\begin{array}{c c} 
						\beta & 0 \\
						0 & 1 \\
					\end{array}  & 	\begin{array}{c c c c} 
						0 & 0 & \cdots & 0 \\
						0 & 0 & \cdots & 0 \\
					\end{array} \\ 
					\hline 
					\begin{array}{c c} 
						0 & b_3 \\
						0 & b_4 \\
						\vdots & \vdots \\
						0 & b_n \\
					\end{array} 
					& B 
				\end{array} 
				\right)
				\]$$
			\end{equation}
			where $\beta \neq 0$ and $B \in \operatorname{GL}_{n-2}(\F)$.
			\item[(iii)] The Leibniz algebra of biderivations of $L_n$ consists of the pairs $(d,D)$ of linear endomorphisms of $L_n$ which are represented by the pair of matrices
			\begin{equation}
				$$\[
				\left( \left(
				\begin{array}{c | c} 
					\begin{array}{c c} 
						\alpha & 0 \\
						0 & 0 \\
					\end{array}  & 	\begin{array}{c c c c} 
						0 & 0 & \cdots & 0 \\
						0 & 0 & \cdots & 0 \\
					\end{array} \\ 
					\hline 
					\begin{array}{c c} 
						0 & a_3 \\
						0 & a_4 \\
						\vdots & \vdots \\
						0 & a_n \\
					\end{array} 
					& A 
				\end{array} 
				\right), \left(
				\begin{array}{c | c} 
					\begin{array}{c c} 
						0 & \alpha' \\
						0 & 0 \\
					\end{array}  & 	\begin{array}{c c c c} 
						0 & 0 & \cdots & 0 \\
						0 & 0 & \cdots & 0 \\
					\end{array} \\ 
					\hline 
					\begin{array}{c c} 
						0 & a_3' \\
						0 & a_4' \\
						\vdots & \vdots \\
						0 & a_n' \\
					\end{array} 
					& A' 
				\end{array} 
				\right) \right)
				\]$$
			\end{equation}
			where $A$,$A' \in \operatorname{M}_{n-2}(\F)$.
		\end{enumerate} \qed
	\end{thm}
	
	\section{The integration of the Leibniz algebra $L_n$}
	
	The \emph{coquecigrue problem} is the problem formulated by J.-L.\ Loday in \cite{loday1993version} of finding a generalization of Lie third theorem to Leibniz algebras. Given a real Leibniz algebra $L$, one wants to find a manifold endowed with a smooth map, which plays the role of the adjoint map for Lie groups, such that the tangent space at a distinguished element, endowed with the differential of this map, gives a Leibniz algebra isomorphic to $L$. Moreover, when $L$ is a Lie algebra, we want to obatin the simply connected Lie group associated with $L$. From now on, we assume that the underlying field of any algebra is $\F=\R$.\\
	
	In \cite{kinyon2004leibniz} M.\ K.\ Kinyon shows that it is possible to define an algebraic structure, called \emph{rack}, whose operation, differentiated twice, defines on its tangent space at the unit element a Leibniz algebra structure.
	
	\begin{definition}
		A \emph{rack} is a set $X$ with a binary operation $\rhd \colon X \times X \rightarrow X$ which is left autodistributive
		\[
		x \rhd (y \rhd z) = (x \rhd y) \rhd (x \rhd z), \quad \forall x,y,z \in X
		\]
		and such that the left multiplications $x \rhd -$ are bijections.
		
		A rack is \emph{pointed} if there exists an element $1 \in X$ such that $1 \rhd x = x$ and $x \rhd 1 = 1$, for any $x \in X$.
		
		A rack is a \emph{quandle} if the binary operation $\rhd$ is idempotent.
	\end{definition}

The first example of a rack is any group $G$ endowed with its conjugation
\[
x \rhd y = x y x^{-1}, \quad \forall x,y \in G.
\]
We denote this rack by $\operatorname{Conj}(G)$ and we observe that it is a quandle.

\begin{definition}
	A pointed rack $(X, \rhd, 1)$ is said to be a \emph{Lie rack} if $X$ is a smooth manifold, $\rhd$ is a smooth map and the left multiplications are diffeomorphisms.
\end{definition}

M.\ K.\ Kinyon proved that the tangent space $\operatorname{T}_1 X$ at the unit element $1$ of a Lie rack $X$, endowed with the bilinear operation
\begin{equation*}
 [x,y] = \frac{\partial^2}{\partial s\partial t}\bigg|_{s,t=0} \gamma_1(s) \rhd \gamma_2(t)
\end{equation*}
where $\gamma_1,\gamma_2 \colon [0,1] \rightarrow X$ are smooth paths such that $\gamma_1(0)=\gamma_2(0)=1$, $\gamma_1'(0)=x$ and $\gamma_2'(0)=y$, is a Leibniz algebra.

He also solved the coquecigrue problem for the class of \emph{split Leibniz algebras}. Here a Leibniz algebra is said to be \emph{split} if there exists an ideal 
\[\operatorname{Leib}(L) \subseteq I \subseteq \operatorname{Z}_l (L)
\]
and a Lie subalgebra $M$ of $L$ such that $L \cong (M \oplus I, \lbrace -,- \rbrace)$, where the bilinear operation $\lbrace -,- \rbrace$ is defined by
\[
\lbrace (x,a),(y,b) \rbrace = ([x,y],\rho_x(b))
\]
and $\rho \colon M \times I \rightarrow I$ is the action on the $M$-module $I$. $L$ is said to be the \emph{demisemidirect product} of $M$ and $I$. More precisely, we have the following.
\begin{thm}\cite{kinyon2004leibniz}
	Let $L$ be a split Leibniz algebra. Then a Lie rack integrating $L$ is $X=(H \oplus I, \rhd)$, where $H$ is the simply connected Lie group integrating $M$ and the binary operation is defined by
	\[
	(g,a) \rhd (h,b)=(ghg^{-1},\phi_g(b)),
	\]
	where $\phi$ is the exponentiation of the Lie algebra action $\rho$.
\end{thm}

Some years later S.\ Covez generalized M.\ K.\ Kinyon's results proving that every real Leibniz algebra admits an integration into a \emph{Lie local rack} (see \cite{covez2013local}). More recently it was showed in \cite{LM2022} that the integration proposed by S.\ Covez is global for any nilpotent Leibniz algebra. Moreover, when a Leibniz algebra $L$ is integrated into a Lie quandle $X$, it turns out that $L$ is a Lie algebra and $X=\operatorname{Conj}(G)$, where $G$ is the simply connected Lie group integrating $L$.\\

Our aim here is to solve the coquecigrue problem for the non-nilpotent Leibniz algebra $L_n=S_2 \oplus \F^{n-2}$. One can check that $S_2$ is a split Leibniz algebra, in the sense of M.\ K.\ Kinyon, with $I=\operatorname{Z}_l(S_2)\cong \R$ and $M \cong \R$. Thus $L \cong (\R^2, \lbrace -,- \rbrace)$ with the bilinear operation defined by
\[
\lbrace (x_1,x_2),(y_1,y_2) \rbrace = (0,\rho_{x_1}(y_2))
\]
and $\rho_{x_1}(y_2)=x_1 y_2$, for any $x_1,y_2 \in \R$. It turns out that a Lie rack integrating $S_2$ is $(\R^2, \rhd)$, where
\[
(x_1,x_2) \rhd (y_1,y_2) = (y_1,y_2 + e^{x_1}y_2).
\]
and the unit element is $(0,0)$. Finally, one can check that the binary operation
\[
(x_1,x_2,x_3,\ldots,x_n) \rhd (y_1,y_2,y_3,\ldots,y_n) = (y_1,y_2 + e^{x_1}y_2,y_3,\ldots,y_n)
\]
defines on $\R^n$ a Lie rack structure with unit element $1=(0,\ldots,0)$, such that $(\operatorname{T}_1 \R^n, \rhd )$ is a Leibniz algebra isomorphic to $L_n$. This result, combined with the ones of \cite[Section 4]{LM2022}, completes the classification of Lie racks whose tangent space at the unit element gives a Leibniz algebra with one-dimensional derived subalgebra.

%

\printbibliography

@article{LM2022,
title = {Two-step nilpotent Leibniz algebras},
author = {La Rosa, G. and Mancini, M.},
journal = {Linear Algebra and its Applications},
volume = {637},
number = {7},
pages = {119-137},
year = {2022},
doi = {https://doi.org/10.1016/j.laa.2021.12.013},
}

@article{Mancini,
	author = {Mancini, M.},
	title = {Biderivations of low-dimensional {L}eibniz algebras},
	journal = {H. Albuquerque, J. Brox, C. Martínez, P. Saraiva (eds.), Non-Associative Algebras and Related Topics. NAART 2020. Springer Proceedings in Mathematics \& Statistics},
	year = {2023},
	volume = {427},
	number = {8},
	pages = {127-136},
	publisher = {Springer Cham},
	doi = {https://doi.org/10.1007/978-3-031-32707-0_8},
}

@article{LaRosaMancini2,
	author = {La Rosa, G. and Mancini, M.},
	title = {Derivations of two-step nilpotent algebras},
	journal = {Communications in Algebra},
	volume = {51},
	number = {12},
	pages = {4928-4948},
	year = {2023},
	doi = {https://doi.org/10.1080/00927872.2023.2222415},
}

@article{CigoliManciniMetere,
	author = {Cigoli, A. S. and Mancini, M. and Metere, G.},
	title = {On the representability of actions of {L}eibniz algebras and {P}oisson algebras},
	year = {2023},
	volume = {66},
	number = {4},
	pages = {998-1021},
	journal = {Proceedings of the Edinburgh Mathematical Society}, 
	doi = {https://doi.org/10.1017/S0013091523000548}
}

@article{solvLeib,
	title = {Classification of some solvable Leibniz algebras},
	author = {Demir, I. and Misra, K. C. and Stitzinger, E.},
	journal = {Algebras and Representation Theory},
	volume = {19},
	pages = {405-417},
	year = {2016},
	doi = {https://doi.org/10.1007/s10468-015-9580-5},
}

@article{4dimLeib,
	author = {Demir, I. and Kailash C. Misra, K. C. and Stitzinger, E.},
	title = {On classification of four-dimensional nilpotent Leibniz algebras},
	journal = {Communications in Algebra},
	volume = {45},
	number = {3},
	pages = {1012-1018},
	year  = {2017},
	doi = {https://doi.org/10.1080/00927872.2016.1172626},
}

@article{5dimLeib,
	author = {Khudoyberdiyev, A. Kh. and Rakhimov, I. S. and Said Husain, Sh. K.},
	title = {On classification of 5-dimensional solvable Leibniz algebras},
	journal = {Linear Algebra and its Applications},
	volume = {457},
	pages = {428-454},
	year = {2014},
    number = {27},
	doi = {https://doi.org/10.1016/j.laa.2014.05.034},
}

@article{6dimLeib,
	author = {Demir, I.},
	title = {Classification of some subclasses of 6-dimensional nilpotent Leibniz algebras},
	journal = {Turkish Journal of Mathematics},
	volume = {44},
	number = {5},
	pages = {1012-1018},
	year  = {2020},
	doi = {https://doi.org/10.3906/mat-2002-69},
}

@article{Katgorodov1,
	title = {The geometric classification of 2-step nilpotent algebras and applications},
	author = {Ignatyev, M. V. and Kaygorodov, I. and Popov, Y.},
	journal = {Revista Matem\'atica Complutense},
	volume = {35},
	number = {3},
	pages = {907-922},
	year = {2022},
	doi = {https://doi.org/10.1007/s13163-021-00411-0},
}

@book{ayupov2019leibniz,
  title={Leibniz Algebras: Structure and Classification},
  author={S. Ayupov and B. Omirov and I. Rakhimov},
  isbn={9781000740004},
  year={2019},
  publisher={CRC Press}
}

@article{NonAbExt,
author = {Liu, J. and Sheng, Y. and Wang, Q.},
title = {On non-abelian extensions of Leibniz algebras},
journal = {Communications in Algebra},
volume = {46},
number = {2},
pages = {574-587},
year  = {2018},
doi = {https://doi.org/10.1080/00927872.2017.1324870},
}

@article{loday1993version,
	title={Une version non commutative des algebres de Lie: les algebres de Leibniz},
	author={Loday, J.-L.},
	journal={L'Enseignement Math\'ematique},
	volume={39},
	number = {3-4},
	pages={269-293},
	year={1993}
}

@article{Falcone2017class,
	title={A class of nilpotent Lie algebras admitting a compact subgroup of automorphisms},
	author={Biggs, R. and Falcone, G.},
	journal={Differential Geometry and its Applications},
	volume={54},
	pages={251--263},
	year={2017},
	publisher={Elsevier},
        doi={https://doi.org/10.1016/j.difgeo.2017.04.009}
}

@article{bartolone2011nilpotent,
	title={Nilpotent Lie algebras with 2-dimensional commutator ideals},
	author={Bartolone, C. and Di Bartolo, A. and Falcone, G.},
	journal={Linear Algebra and its Applications},
	volume={434},
	number={3},
	pages={650--656},
	year={2011},
	doi={http://doi:10.1016/j.laa.2010.09.036}
}

@article{BarDiFal18,
	title={Solvable extensions of nilpotent complex Lie algebras of type $\{$2n,1,1$\}$},
	author={Bartolone, C. and Di Bartolo, A. and Falcone, G.},
	journal={Moscow Mathematical Journal},
	volume={18},
	number={4},
	pages={607--616},
	year={2018},
    doi={http://dx.doi.org/10.17323%2F1609-4514-2018-18-4-607-616}
}

@article{blokhLie,
	title={A generalization of the concept of a Lie algebra},
	author={Blokh, A.},
	journal={Dokl. Akad. Nauk SSSR},
	volume={165},
	number={3},
	pages={471--473},
	year={1965}
}

@article{Dieudonné,
     author = {Dieudonn\'e, J.},
     title = {Sur la r\'eduction canonique des couples de matrices},
     journal = {Bulletin de la Soci\'et\'e Math\'ematique de France},
     pages = {130--146},
     publisher = {Soci\'et\'e math\'ematique de France},
     volume = {74},
     year = {1946},
     doi = {https://doi.org/10.24033/bsmf.1380},
     zbl = {0061.01307},
     mrnumber = {9,264f},
     language = {fr}
}

@book{erdmann2006introduction,
	title={Introduction to Lie Algebras},
	author={Erdmann, K. and Wildon, M. J.},
	isbn={9781846284908},
	year={2006},
	publisher={Springer London}
}

@article{2-dimCuvier,
	author = {Cuvier, C.},
	journal = {Annales scientifiques de l'École Normale Supérieure},
	number = {1},
	pages = {1-45},
	title = {Algèbres de Leibnitz: définitions, propriétés},
	volume = {27},
	year = {1994},
}

@article{covez2013local,
	title={The local integration of Leibniz algebras},
	author={Covez, S.},
	journal={Annales de l'Institut Fourier},
	volume={63},
	number={1},
	pages={1-35},
	year={2013},
	doi = {https://doi.org/10.5802/aif.2754},
}

@article{kinyon2004leibniz,
	title={Leibniz algebras, Lie racks, and digroups},
	author={Kinyon, M. K.},
	journal={Journal of Lie Theory},
	year={2007},
	volume={17},
	number={1},
	pages={99-114},
}

\end{document}